\documentclass{svjour3}
\usepackage{amsmath,amssymb,mathdots}
\usepackage{bm} 
\usepackage{cases,booktabs}
\usepackage{array,threeparttable,multirow}
\usepackage[math]{cellspace}
\usepackage[table]{xcolor}
\usepackage{url}
\usepackage{graphicx,subfigure}
\usepackage{algpseudocode}

\newcommand{\bracket}[1]{\ensuremath{\left[ #1 \right]}}

\newcommand{\pair}[1]{\ensuremath{\langle #1 \rangle}}

\newcommand{\SO}{\operatorname{SO(3)}}
\newcommand{\so}{\ensuremath{\mathfrak{so}(3)}}
\renewcommand{\Re}{\ensuremath{\mathbb{R}}}
\newcommand{\Cp}{\ensuremath{\mathbb{C}}}
\newcommand{\Sph}{\ensuremath{\mathsf{S}}}

\title{Real Harmonic Analysis on the Special Orthogonal Group}
\author{Taeyoung Lee}
\institute{Taeyoung Lee \at Mechanical and Aerospace Engineering\\
    George Washington University\\
    800 22nd St NW,\\ 
    Washington DC 20052.\\
    Tel.: 1-202-994-8710\\
    \email{tylee@gwu.edu}
}
\date{}

\graphicspath{{./Figs/}}

\begin{document}
\allowdisplaybreaks

\maketitle

\begin{abstract}
    This paper presents theoretical analysis and software implementation for real harmonics analysis on the special orthogonal group. 
    Noncommutative harmonic analysis for complex-valued functions on the special orthogonal group has been studied extensively.
    However, it is customary to treat real harmonic analysis as a special case of complex harmonic analysis, and there have been limited results developed specifically for real-valued functions.
    Here, we develop a set of explicit formulas for real-valued irreducible unitary representations on the special orthogonal group, and provide several operational properties, such as derivatives, sampling, and Clebsch-Gordon coefficients. 
    Furthermore, we implement both of complex and real harmonics analysis on the special orthogonal group into an open source software package that utilizes parallel processing through the OpenMP library. 
    The efficacy of the presented results are illustrated by benchmark studies and an application to spherical shape matching.
    \keywords{Fast Fourier transform, special orthogonal group, noncommutative harmonic analysis, spherical harmonics}
    \subclass{22E45, 43A30, 65T99}
\end{abstract}

\section{Introduction}

Noncommutative harmonic analysis is a generalization of Fourier analysis to topological groups that are not necessarily commutative~\cite{GroAMM78}.
According to the Peter-Weyl theorem~\cite{PetWeyMA27}, irreducible unitary matrix representation of a compact group forms an complete orthogonal basis for square-integrable functions on the group. 
Consequently, such functions can be expanded as a linear combination of matrix representations weighted by Fourier parameters. 
In particular, harmonic analysis on the special orthogonal group, or the rotation group, has been utilized in quantum physics~\cite{VarMos88,BieLou81}.
Recently, it also has been applied to various engineering problems in robotics, controls, and machine learning~\cite{ChiKya01,LeeAJDSMC15,CohGeiARX18}. 
Computationally efficient numerical implementation and fast Fourier transform algorithms on the special orthogonal group are considered in~\cite{RisJG96,KosRocJFAA08}.

Despite extensive prior works in broad areas of science and engineering, all of the aforementioned references deal with complex harmonic analysis, where matrix representations and Fourier-parameters are complex-valued. 
There have been limited efforts in noncommutative harmonic analysis for real-valued functions on the special orthogonal group.   
One of the earlier studies in real harmonic analysis include~\cite{SheGriJMS86}, where real-valued matrix representations for several atomic orbitals are evaluated up to a certain order. 
Later, in~\cite{IvaRueJPC96,IvaRueJPC98}, recurrence relations to evaluate real matrix representations are presented in terms of nine elements of a rotation matrix. 
In~\cite{BlaFloJMS97}, real-valued matrix representations are constructed by transforming complex-valued matrix representations, namely wigner D-matrices.

In this paper, we develop a new form of matrix representations for real harmonic analysis on the special orthogonal group. 
Compared with~\cite{IvaRueJPC96}, these are formulated in terms of Euler angles so that a real fast Fourier transform on the special orthogonal group can be developed utilizing various discrete fast Fourier transform algorithms in the Euclidean space. 
The presented form is based on the wigner d-function of the second Euler angle, which can be computed by a recurrence relation. 
But, once the wigner d-function is evaluated, it is explicit with respect to the remaining two other Euler angles. 
Compared with~\cite{BlaFloJMS97}, the presented real-valued matrix representation can be constructed without need for evaluating complex-valued matrix representation, or the wigner D-matrices.

Furthermore, we present several operational properties of real harmonic analysis. 
First, a sample theorem is presented such that Fourier transform of a band-limited function with a bandwidth $B$ can be exactly computed with $(2B)^3$ function evaluations. 
As such, there is no need to approximate an integration over the special orthogonal group with a quadrature rule, when computing Fourier transforms. 
This is utilized to construct a fast Fourier transform for real-valued functions on the special orthogonal group. 
Next, we develop Clebsch-Gordon coefficients for real harmonics so that a product of two real matrix representations is written as a linear combination of other real matrix representations. 
Finally, we show an explicit expression for the derivatives of real matrix representations to formulate representations for lie algebra. 

Beyond these theoretical contributions, an open source software package has been developed for real harmonic analysis on the special orthogonal group. 
This library includes fast Fourier transform, inverse Fourier transform, and evaluation of Clebsch-Gordon coefficients and derivatives. 
While this paper focuses on real harmonic analysis on the special orthogonal group, this software library also provides complex harmonic analysis, including wigner D-functions, and spherical harmonics on the unit-sphere as well.  
There are developed in \texttt{c++} while utilizing the OpenMP library~\cite{DagMenICSE98} to support accelerated parallel computing for multithread processors. 
These are verified by various software unit-testing.
A benchmark study and a particular application to spherical shape matching for Earth topological data are presented as well. 

In short, the main contribution of this paper is theoretical and numerical analysis for the foundation of real harmonic analysis on the special orthogonal group. 
The presented form of real matrix representations and its derivatives, and real Clebsch-Gordon coefficients have been unprecedented. 
Also, the proposed software library can be utilized in various engineering application of noncommutative harmonic analysis on the special orthogonal group.

\section{Complex Harmonic Analysis on $\SO$}\label{sec:CHA}

The three-dimensional special orthogonal group is composed of $3\times 3$ orthogonal matrices with determinant one, i.e.,
\begin{equation}
    \SO = \{ R\in\Re^{3\times 3}\,|\, R^T R= I_{3\times 3},\; \mathrm{det}[R]=1\}.
\end{equation}
Harmonic analysis for complex-valued functions on $\SO$ has been studied extensively, for example in~\cite{ChiKya01,VarMos88,MarPec11}. 
In this section, we summarize selected results that are used for the subsequent  developments of real harmonics on $\SO$.

\subsection{Euler Angles}

Any $R\in\SO$ can be parameterized by three angles $\alpha,\gamma\in[0,2\pi)$ and $\beta\in[0,\pi]$ as
\begin{equation}
    R(\alpha,\beta,\gamma) = \exp(\alpha \hat e_3) \exp(\beta\hat e_2) \exp(\gamma\hat e_3),
\end{equation}
where $e_2=(0,1,0)$, $e_3=(0,0,1)\in\Re^3$, and the hat map $\hat\cdot:\Re^3\rightarrow\so$ is defined such that $x\times y = \hat x y$ for any $x,y\in\Re^3$. 
The Lie algebra is defined as $\so=\{S\in\Re^{3\times 3}\,|\, S+S^T=0\}$.
These are referred to as 3--2--3 Euler angles. 
While this parameterization involves complicated combinations of trigonometric functions and inherent singularities in the resulting kinematics equation, there are several advantages in harmonic analysis, such as convenience in applying fast Fourier transform techniques. 

\subsection{Irreducible Unitary Representation: Wigner $D$-Matrix}

As a transformation group, $\SO$ acts on $\Re^3$ while preserving its inherent structures of $\Re^3$.  
For example, $Rx\in\Re^3$ for any $R\in\SO$ and $x\in\Re^3$ via the matrix multiplication.  
Let $\mathcal{L}^2(\Re^3)$ be the set of complex-valued, square-integrable functions on $\Re^3$.
The left regular representation on $\SO$, namely $D(R)$ for $R\in\SO$ is an operator $D(R): \mathcal{L}^2(\Re^3)\rightarrow \mathcal{L}^2(\Re^3)$ defined such that
\begin{equation}
    (D(R) f)(x)=f( R^T x),\label{eqn:DRf}
\end{equation}
for any $f\in \mathcal{L}^2(\Re^3)$ and $x\in\Re^3$, i.e., it transforms a function such that it becomes equivalent to rotating the input argument by $R^T$. 
It is straightforward to show that $D(\cdot)$ is a \textit{homomorphism}, i.e., $D(R_1)D(R_2)=D(R_1R_2)$ for any $R_1,R_2\in\SO$.
Also, it is a linear operator on $L^2(\Re^3)$. 
Consequently, by selecting a basis on $\mathcal{L}^2(\Re^3)$, $D(R)$ can be represented with a matrix, thereby resulting in \textit{matrix representations}. 

The matrix representation on $\SO$ is often induced from that of $\mathrm{SU(2)}=\{U\in\Cp^{2\times 2}\,|\, U^* U =I_{2\times 2}, \mathrm{det}[U]=1\}$ utilizing one-to-two correspondence between $\SO$ and $\mathrm{SU(2)}$. 
More specifically, consider representation on $\mathrm{SU(2)}$ as an operator on the analytic functions from $\Cp^2$ to $\Cp$. 
By selecting the set of homogeneous polynomials as the basis of the analytic functions, one can derive matrix representation of $\mathrm{SU}(2)$ as
\begin{equation}
    D^l_{m,n}(R(\alpha,\beta,\gamma)) = e^{-im\alpha} d^l_{m,n} (\beta) e^{-in\gamma},\label{eqn:D}
\end{equation}
which is referred to as the \textit{wigner D-matrix} that is common in quantum mechanics~\cite{VarMos88}.
Here the range of the index $l$ is $\{0,\frac{1}{2},1,\frac{3}{2}\ldots \}$, and for each $l$, the indices $m,n$ vary in $\{-l,-l+1,\ldots, l-1, l\}$. 
In \eqref{eqn:D}, the real-valued $d^l_{m,n}$ is called \textit{wigner d-matrix}. 
An explicit form of $d^l_{m,n}$ is tabulated up to $l=5$ in~\cite{VarMos88}, and a recursive algorithm to evaluate it for arbitrary order is presented in~\cite{BlaFloJMS97}.
The above expression for matrix representation $\mathrm{SU(2)}$ results in a matrix representation of $\SO$ by taking the integer values of $l$ only.
Throughout this paper, $D^l(R)\in\Cp^{(2l+1)\times (2l+1)}$ is considered as a square matrix where the row index and the column index are $m$ and $n$, respectively, which vary from $-l$ to $l$ in ascending order. 

From \eqref{eqn:DRf}, it is trivial to show $D^l(I_{3\times 3})= I_{2l+1\times 2l+1}$, or $D^l_{m,n}(I_{3\times 3})=\delta_{m,n}$. 
This representation is \textit{irreducible}, i.e., cannot be block diagonalized consistently with a similarity transform, and  it is \textit{unitary}, i.e., 
\[
    (D^l(R))^*= (D^l(R))^{-1}= D^l(R^T),
\]
where the last equality is from the homomorphism property.
When $\alpha=\gamma=0$, these also imply
\begin{equation}\label{eqn:d_properties}
    d^l(-\beta)=(d^l(\beta))^{-1}= (d^l(\beta))^T,\quad d^l(0)=I_{2l+1\times 2l+1}, \quad d^l_{m,n}(0)=\delta_{m,n}.
\end{equation}
While this paper follows the convention of 3--2--3 Euler angles, using other types of Euler angles in fact does not matter as it would yield an equivalent form of matrix representations that can be constructed by a similarity transform of \eqref{eqn:D}.

\subsection{Fourier Transform on $\SO$}

According to Peter-Weyl theorem, the irreducible, unitary representations form a complete orthogonal basis for $\mathcal{L}^2(\SO)$~\cite{PetWeyMA27}. 
Consequently, any $f\in\mathcal{L}^2(\SO)$ is expanded into
\begin{align}
    f(R(\alpha,\beta,\gamma)) &= \sum_{l=0}^\infty \sum_{m,n=-l}^l (2l+1) F^l_{m,n} D^l_{m,n}(\alpha,\beta,\gamma),\label{eqn:f_IFT}
\end{align}
for Fourier parameters $F^l_{m,n}\in\Cp$.

Define an inner product on $\mathcal{L}^2(\SO)$ as 
\begin{equation}
    \pair{f(R), g(R)} = \int_{\SO} \overline {f(R)} g(R) dR,
\end{equation}
where $dR$ is the Haar measure on $\SO$ that is normalized such that $\int_{\SO} dR =1$. 
For example, using 3--2--3 Euler angles, it is given by $dR = \frac{1}{8\pi^2} \sin\beta d\alpha d\beta d\gamma$. 
The orthogonality of the irreducible, unitary representation implies 
\begin{equation}
    \pair{ D^{l_1}_{m_1,n_1}(R), D^{l_2}_{m_2,n_2}(R)} = \frac{1}{2l_1+1}\delta_{l_1,l_2}\delta_{m_1,m_2}\delta_{n_1,n_2}. \label{eqn:D_ortho}
\end{equation}
Therefore, the Fourier parameters in \eqref{eqn:f_IFT} can be discovered by the inner product,
\begin{equation}
    F^l_{m,n}= \pair{ D^l_{m,n}(R), f(R)}.\label{eqn:FT}
\end{equation}
Equations \eqref{eqn:FT} and \eqref{eqn:f_IFT} are considered as the Fourier transform of $f(R)$, and its inverse transform, respectively. 

A function $f\in\mathcal{L}(\SO)$ is \textit{band-limited with the band $B$}, if its Fourier parameters vanish, i.e., $F^l_{m,n}=0$ for any $l\geq B$. 
The classical sampling theorem states that the Fourier transform of a band-limited function can be recovered from the sample values that are chosen at a uniform grid with a frequency that is at least twice of the band-limit. 
Using this, a fast Fourier transform technique is presented in~\cite{KosRocJFAA08}. 

\subsection{Derivatives of Representation}

Given the matrix representation $D^l(R)$ on the Lie group $\SO$, the representation on the Lie algebra $\so\simeq \Re^3$ is constructed by differentiation.
More specifically, the $l$-th representation, namely $u^l:\Re^3\rightarrow \Cp^{(2l+1)\times (2l+1)}$ is given by
\begin{equation}
    u^l(\eta)= \frac{d}{d\epsilon}\bigg|_{\epsilon=0} D^l(\exp(\epsilon\hat\eta)),\label{eqn:ul}
\end{equation}
for $\eta\in\Re^3$. 
Being a Lie algebra homogeneous, it satisfies $u^l(\eta\times\zeta)=[u^l(\eta),u^l(\zeta)]$, where $\eta,\zeta\in\Re^3$ and $[\cdot,\cdot]$ represents the matrix commutator.

Since it is a linear operator, $u^l(\eta)$ for an arbitrary $\eta\in\Re^3$ is given by a linear combination of $u^l(e_i)$ for $i=\{1,2,3\}$.
Since $\exp(\epsilon\hat e_3) = R(\epsilon,0,0)$,
\begin{align}
    u^l_{m,n}(e_3) & = \frac{d}{d\epsilon}\bigg|_{\epsilon=0} e^{-im\epsilon}d^l_{m,n}(0)=-im\delta_{m,n}.
\end{align}
Similarly, using $\exp(\epsilon\hat e_2) = R(0,\epsilon,0)$, 
\begin{align}
    u^l_{m,n}(e_2) 
    & = -\frac{1}{2}\sqrt{(l+m)(l-m+1)} \delta_{m-1,n}\nonumber\\
    & \quad + \frac{1}{2}\sqrt{(l-m)(l+m+1)} \delta_{m+1,n}.
\end{align}
Last, $u^l(e_1)$ can be obtained by $u^l(e_1)=u^l(e_2\times e_3)=[u^l(e_2), u^l(e_3)]$ as
\begin{align}
    u^l_{m,n}(e_1) 
    & = -\frac{1}{2}i \sqrt{(l+m)(l-m+1)} \delta_{m-1,n}\nonumber\\
    &\quad-\frac{1}{2}i \sqrt{(l-m)(l+m+1)} \delta_{m+1,n}.
\end{align}
These results are useful in engineering application of harmonic analysis, particularly when computing the sensitivity of \eqref{eqn:f_IFT} with respect to $R$.
Similar expressions are presented in~\cite{ChiKya01}, but for a different form of matrix representations.

\subsection{Clebsch-Gordon coefficients}

The product of two wigner D-matrices can be rewritten as a finite linear combination of other wigner D-matrices as follows. 
\begin{equation}
    D^{l_1}_{m_1,n_1} (R) D^{l_2}_{m_2,n_2}(R) = \sum_{l=|l_1-l_2|}^{l_1+l_2} \sum_{m,n=-l}^l C^{l,m}_{l_1,m_1,l_2,m_2} C^{l,n}_{l_1,n_1,l_2,n_2} D^{l}_{m,n}(R).\label{eqn:D1D2_0}
\end{equation}
Interestingly, the coupling coefficients are split into two parts as written above, and they are referred to as Clebsch-Gordon coefficients.
There are several properties of Clebsch-Gordon coefficients listed in~\cite{VarMos88}, including
\begin{gather}
    \sum_{m_1,m_2} C^{l,m}_{l_1,m_1,l_2,m_2} C^{l',m'}_{l_1,m_1,l_2,m_2} = \delta_{l,l'}\delta_{m,m'},\label{eqn:C_ortho_0}\\
    \sum_{l,m} C^{l,m}_{l_1,m_1,l_2,m_2} C^{l,m}_{l_1,m'_1,l_2,m'_2} = \delta_{m_1,m'_1}\delta_{m_2,m'_2},\label{eqn:C_ortho_1}\\
    C^{l,m}_{l_1,m_1,l_2,m_2} = (-1)^{l_1+l_2-l} C^{l,-m}_{l_1,-m_1,l_2,-m_2}, \label{eqn:C_sym}\\
    C^{l,m}_{l_1,m_1,l_2,m_2} = 0 \quad \text{ if } m \neq m_1 + m_2.
\end{gather}
Using the last property, the double summation at \eqref{eqn:D1D2_0} reduces to  
\begin{equation}
    D^{l_1}_{m_1,n_1} (R) D^{l_2}_{m_2,n_2}(R) = \sum_{l=\underline{l}}^{l_1+l_2} C^{l,m_1+m_2}_{l_1,m_1,l_2,m_2} C^{l,n_1+n_2}_{l_1,n_1,l_2,n_2} D^{l}_{m_1+m_2,n_1+n_2}(R).\label{eqn:D1D2}
\end{equation}
for $\underline{l}=\max\{|l_1-l_2|,|m_1+m_2|,|n_1+n_2|\}$.
While \eqref{eqn:D1D2_0} commonly appears in the literature, \eqref{eqn:D1D2} is not available at least in the cited references. 
A computational scheme to evaluate Clebsch-Gordon coefficients is proposed in~\cite{Str14}. 

In~\cite{MarPec11}, it is proposed to rearranged the coefficients into a matrix $C_{l_1,l_2}\in \Re^{(2l_1+1)(2l_2+1)\times(2l_1+1)(2l_2+1)}$ according to the following ordering rules: 
\begin{align}
    (\text{column index of } C^{l,m}_{l_1,m_1,l_2,m_2}) &= l^2-(l_2-l_1)^2 + l + m,\label{eqn:order_scheme_c}\\
    (\text{row index of } C^{l,m}_{l_1,m_1,l_2,m_2}) &= (l_1+m_1)(2l_2+1) + l_2+m_2,\label{eqn:order_scheme_r}
\end{align}
which begins from $0$ following the convention of the C programming language that is adopted for software implementation in this paper. 
Under this matrix formulation, \eqref{eqn:D1D2_0} is rearranged into
\begin{equation}
    D^{l_1}(R) \otimes D^{l_2}(R) = C_{l_1,l_2} \bracket{ \bigoplus_{l=|l_1-l_2|}^{l_1+l_2} D^l(R)} C_{l_1,l_2}^T,\label{eqn:Clebsch_Gordon}
\end{equation}
where $\otimes$ denotes the Kronecker product. 

\subsection{Relation to Spherical Harmonics}

Consider the unit-sphere, $\Sph^2=\{x\in\Re^3\,|\, \|x\|=1\}$. 
Let $x\in\Sph^2$ be parameterized by the co-latitude $\theta\in[0,\pi]$ and the longitude $\phi\in[0,2\pi)$ as $ x(\theta,\phi)=[\cos\phi\sin\theta, \sin\phi\sin\theta, \cos\theta]$. 
Spherical harmonics, namely $Y^l_m(\theta,\phi)$ is defined as
\begin{equation}
    Y^l_m(\theta,\phi) = e^{im\phi} \sqrt{\frac{2l+1}{4\pi}\frac{(l-m)!}{(l+m)!}}  P^m_l(\cos\theta),
\end{equation}
with the associated Legendre polynomials $P^m_l$ for $l\in\{0,1,\ldots\}$ and $-l\leq m\leq l$.

Let $dx=\frac{1}{4\pi} \sin\theta d\phi d\theta$ be the measure of $\Sph^2$, normalized such that $\int_{\Sph^2} dx =1$. 
We define an inner product on the square-integrable functions, namely $\mathcal{L}^2(\Sph^2)$ as
\begin{align*}
    \pair{ f(x), g(x) }_{\mathcal{L}(\Sph^2)} = \int_{\Sph^2} \overline{f(x)}g(x) dx.
\end{align*}
Spherical harmonics satisfies the following orthogonality with respect to the above inner product, yielding
\begin{equation}
    \pair{ Y^{l_1}_{m_1}(x), Y^{l_2}_{m_2}(x)}_{\mathcal{L}(\Sph^2)} = \frac{1}{4\pi} \delta_{l_1l_2} \delta_{m_1m_2}.\label{eqn:Y_ortho}
\end{equation}

Spherical harmonics are closely related to the wigner D-function.
Using the homomorphism property of the group representation~\cite{ChiKya01}, we obtain
\begin{equation}
    Y^{l}_{n}(R^T x) = \sum_{m'} Y^{l}_{m'}(x) D^l_{m',n}(R) .\label{eqn:YRx}
\end{equation}
Therefore, using \eqref{eqn:Y_ortho}, the wigner D-function can be rediscovered from the spherical harmonics as
\begin{equation}
    \pair{Y^{l}_{m}(x), Y^{l}_{n}(R^T x) }_{\mathcal{L}(\Sph^2)} 
       = \frac{1}{4\pi} D^l_{m,n}(R).
\end{equation}
Let $Y^l$ be the $(2l+1)\times 1$ column vector whose elements are composed of $Y^l_m$ for $m\in\{-l,\ldots l\}$ in  ascending order. 
The above equation is rewritten in a matrix form as
\begin{equation}\label{eqn:D_Y}
    \pair{ Y^l(x), (Y^l(R^Tx))^T }_{\mathcal{L}(\Sph^2)} = \frac{1}{4\pi} D^l(R).
\end{equation}

\section{Real Harmonic Analysis on $\SO$}

The objective of this paper is to develop the counterparts of Section 2 for real-valued functions on $\SO$, resulting in real harmonic analysis on $\SO$. 
Instead of formulating as a specialized form of wigner D-matrices, real matrix representations are directly constructed in terms of Euler angles, 
and they are utilized for fast Fourier transform of real-valued functions on $\SO$. 
Further, Clebsch-Gordon coefficients and derivatives are formulated as well. 

\subsection{Real Irreducible Unitary Representations}

We follow the approaches presented in~\cite{BlaFloJMS97}, where real harmonics on $\SO$ is constructed from real spherical harmonics. 
Orthogonal basis for real-valued functions on $\Sph^2$, namely $S^l(x)\in\Re^{2l+1}$ is constructed by the following  transform,
\begin{equation}\label{eqn:YtoS}
    S^l(x) = T^l Y^l(x),
\end{equation}
where $x\in\Sph^2$ and the matrix $T^l\in\Cp^{(2l+1)\times(2l+1)}$ is defined as
\begin{equation}\label{eqn:Tmn}
    T_{m,n} = 
    \begin{cases}
        1 & m=n=0,\\
        0 & |m|\neq |n|,\\
        \frac{(-1)^m}{\sqrt{2}} & m>0, n=m,\\
        \frac{1}{\sqrt{2}} & m>0, n=-m,\\
        \frac{i}{\sqrt{2}} & m<0, n=m,\\
        \frac{-i(-1)^m}{\sqrt{2}} & m<0,n=-m,
    \end{cases}
\end{equation}
or in a matrix form,
\begin{equation}\label{eqn:T}
    T^l=\frac{1}{\sqrt{2}}
    \begin{bmatrix}
        i & 0 & \cdots & 0 & \cdots & 0 & -i(-1)^l\\
        0 & i & \cdots & 0 & \cdots & -i(-1)^{l-1} & 0 \\
        \vdots & \vdots & \ddots &\vdots & \iddots & \vdots & \vdots\\
        0 & 0 & \hdots & \sqrt{2} & \hdots & 0 & 0 \\
        \vdots & \vdots & \iddots &\vdots & \ddots & \vdots & \vdots\\
        0 & 1 & \cdots & 0 & \cdots & (-1)^{l-1} & 0 \\
        1 & 0 & \cdots & 0 & \cdots & 0 & (-1)^l\\
    \end{bmatrix}.
\end{equation}
It is straightforward to show that the columns of $T^l$ are mutually orthonormal, i.e., $T^l$ is unitary so that  
\begin{equation}
    (T^l)^{-1} = (T^l)^* = (\overline{T^l})^T.\label{eqn:T_unitary}
\end{equation}

Let $U^l\in\Re^{(2l+1)\times(2l+1)}$ be the matrix for the $l$-th real harmonics on $\SO$.
Motivated by \eqref{eqn:D_Y}, it is defined as
\begin{equation}\label{eqn:U_S}
    \pair{ S^l(x), (S^l(R^Tx))^T } = \frac{1}{4\pi} U^l(R).
\end{equation}
Substituting \eqref{eqn:YtoS} and rearranging, 
\begin{equation}
    U^l (R) = \overline{T^l} D^l(R) (T^l)^T.\label{eqn:Ul_T}
\end{equation}
This is a homomorphism, i.e., $U^l(R_1R_2)=U^l(R_1)U^l(R_2)$ for any $R_1,R_2\in\SO$. 
Furthermore, it is irreducible and orthogonal, i.e.,  $(U^l(R))^{-1} = (U^l(R))^T = U^l(R^T)$, where the last equality is from the homomorphism property and $U^l(I_{3\times 3})= I_{(2l+1)\times(2l+1)}$.  
Also, similar with \eqref{eqn:YRx} and \eqref{eqn:D_ortho},
\begin{gather}
    S^{l}_{n}(R^T x) = \sum_{m'} S^{l}_{m'}(x) U^l_{m',n}(R) ,\label{eqn:SRx}\\
    \pair{ U^{l_1}_{m_1,n_1}(R) , U^{l_2}_{m_2,n_2} (R) } =  \frac{1}{2l_1+1} \delta_{l_1,l_2}\delta_{m_1,m_2} \delta_{n_1,n_2}.\label{eqn:U_ortho}
\end{gather}

While $U^l(R)$ can be evaluated by transforming $D^l(R)$ according to \eqref{eqn:Ul_T} as in~\cite{BlaFloJMS97}, the procedure will involve unnecessary steps with complex variables. 
Here, we present an alternative, explicit formulation as follows. 
\begin{theorem}
    Real harmonics on $\SO$ defined in~\eqref{eqn:Ul_T} is equivalent to the following formulations. 
    \begin{align}\label{eqn:U_mn}
        U^l_{m,n} (R) =
        \begin{cases}
            -\sin m\alpha \sin n\gamma \Psi^l_{-m,n} (\beta) + \cos m\alpha \cos n\gamma \Psi^l_{m,n}(\beta) \\
            \hspace*{0.3\textwidth} (m\geq 0, n\geq 0) \text{ or } (m<0,n<0),\\
            -\sin m\alpha \cos n\gamma \Psi^l_{-m,n} (\beta) + \cos m\alpha \sin n\gamma \Psi^l_{m,n}(\beta) \\
            \hspace*{0.3\textwidth} (m\geq 0, n < 0) \text{ or } (m<0,n\geq0),
        \end{cases}
    \end{align}
    where
    \begin{align}\label{eqn:Psi}
        \Psi^l_{m,n}(\beta) & = 
        \begin{cases}
            (-1)^{m-n} d^l_{|m|,|n|}(\beta) + (-1)^m\mathrm{sgn}(m) d^l_{|m|,-|n|}(\beta) & mn\neq 0,\\
            (-1)^{m-n} \sqrt{2} d^l_{|m|,|n|}(\beta) & m = 0 \text{ xor } n=0,\\
            d^l_{0,0}(\beta) & m=n=0,
        \end{cases}
    \end{align}
    satisfying $\Psi^l_{m,n}=\Psi^l_{m,-n}$. Or in a matrix formulation,
    \begin{equation}\label{eqn:Ul}
        U^l (R) = X^l(\alpha) W^l(\beta) X^l(\gamma),
    \end{equation}
    where $X^l(\alpha), W^l(\beta)\in\Re^{(2l+1)\times(2l+1)}$ are defined as
    \begin{align}
        X^l_{m,n}(\alpha) & = \begin{cases}
            0 & |m|\neq |n|,\\
            1 & m=n=0,\\
            \cos m\alpha & m=n\neq 0,\\
            -\sin m\alpha & m=-n\neq 0.
        \end{cases}\label{eqn:X}\\
        W^l_{m,n}(\beta) & = \begin{cases}
            \Psi^l_{m,n}(\beta) & (m\geq 0,n\geq 0) \text{ or } (m<0,n<0)\\
            0 & \text{otherwise}.
        \end{cases}\label{eqn:W}
    \end{align}
\end{theorem}

\begin{proof}
Rearranging the matrix multiplication in \eqref{eqn:Ul} into element-wise operations,
\begin{equation*}
    U^l_{m,n}(R) = \sum_{p,q=-l}^l \overline{T}^l_{m,p} T^l_{n,q} D^l_{p,q} (R).
\end{equation*}
From \eqref{eqn:Tmn}, the expression in the summation vanishes when $|p|\neq|m|$ or $|q|\neq|n|$. 
Consequently, the above double summation reduces to the following cases, depending on  whether any sub index is zero or not,
\begin{align*}
    U^l_{0,0} (R) &= D^l_{0,0} (R) = d^l_{0,0}(R),\\
    U^l_{m,0} (R) &= \overline{T}^l_{m,m} D^l_{m,0} (R) + \overline{T}^l_{m,-m} D^l_{-m,0} (R),\\
    U^l_{0,n} (R) &= T^l_{n,n} D^l_{0,n}(R) + T^l_{n,q-n} D^l_{0,-n}(R),\\
    U^l_{m,n} (R) &= \overline{T}^l_{m,m} T^l_{n,n} D^l_{m,n} (R) + \overline{T}^l_{m,m} T^l_{n,-n} D^l_{m,-n} (R)\\
                  &\quad + \overline{T}^l_{m,-m} T^l_{n,n} D^l_{-m,n} (R) + \overline{T}^l_{m,-m} T^l_{n,-n} D^l_{-m,-n} (R).
\end{align*}
for non zero $m,n\in\{-l,\ldots l\}$. 
Throughout the remainder of this proof, we focus on the last case assuming $m,n>0$. 
The results for other cases can be obtained in a similar manner. 
Substituting the values of $T^l_{m,n}$ given in \eqref{eqn:Tmn} for $m,n>0$,
\begin{equation*}
    U^l_{m,n} (R) = \frac{(-1)^{m+n}}{2}  D^l_{m,n} (R) + \frac{(-1)^m}{2}  D^l_{m,-n} (R) +  \frac{(-1)^n}{2} D^l_{-m,n} (R) + \frac{1}{2}  D^l_{-m,-n} (R).
\end{equation*}
We substitute \eqref{eqn:D}, and rearrange it using the symmetry of the wigner d-functions, namely $d^l_{m,n}(\beta) = (-1)^{m-n} d^l_{-m,-n}$, to obtain 
\begin{equation*}
   U^l_{m,n} (R) = (-1)^{m+n} d^l_{m,n}(\beta) \cos (m\alpha + n\gamma) + (-1)^m d^l_{m,-n}(\beta) \cos(m\alpha - n\gamma).
\end{equation*} 
From the definition of $\Psi^l_{m,n}(\beta)$ in \eqref{eqn:Psi} for the case of $m,n>0$ considered here, 
\begin{align*}
    U^l_{m,n} (R) & = \cos m\alpha\cos n\gamma \{(-1)^{m+n} d^l_{m,n}(\beta) +(-1)^m d^l_{m,-n}(\beta)\} \\
    &\quad + \sin m\alpha\sin n\gamma \{-(-1)^{m+n} d^l_{m,n}(\beta)+(-1)^m d^l_{m,-n}(\beta) \} ,
\end{align*}
where the two expression in the braces reduce to $\Psi^l_{m,n}(\beta)$ and $-\Psi^l_{-m,n}(\beta)$, respectively, while yielding \eqref{eqn:U_mn} for $m,n>0$. 
The other remaining cases for \eqref{eqn:U_mn} can be shown in the similar way. 

Next, the matrix product in~\eqref{eqn:Ul} is written as
\begin{equation*}
    U^l_{m,n}(R) = \sum_{p,q=-l}^l X^l_{m,p}(\alpha) W^l_{p,q}(\beta) X^l_{q,n} (\beta).
\end{equation*}
Again, suppose $m,n>0$. 
From \eqref{eqn:X}, the expression in the summation does not vanish only if $p=\pm m$ and $q=\pm n$.
Consequently,
\begin{align*}
    U^l_{m,n}(R) & = X^l_{m,m}(\alpha) W^l_{m,n} (\beta) X^l_{n,n}(\beta) + X^l_{m,m}(\alpha) W^l_{m,-n} (\beta) X^l_{-n,n}(\beta) \\
                 & \quad + X^l_{m,-m}(\alpha) W^l_{-m,n} (\beta) X^l_{n,n}(\beta) + X^l_{m,-m}(\alpha) W^l_{-m,-n} (\beta) X^l_{-n,n}(\beta).
\end{align*}
Substituting \eqref{eqn:X} and \eqref{eqn:W} and using $\Psi^l_{-m,-n}=\Psi^l_{-m,n}$, it is straightforward to show the above reduces to \eqref{eqn:U_mn} for $m,n>0$. 
The other cases can be shown similarly.
$ $\hfill $\blacksquare$ 
\end{proof}

This theorem states that real matrix representation on the special orthogonal group can be constructed directly in terms of Euler angles without need for evaluating complex, wigner D-matrices. 
The expression presented in~\eqref{eqn:U_mn} is composed of sine and cosine terms for multiples of $\alpha,\gamma$, that are similar to those appear in real spherical harmonics, and $\Psi$ terms defined by the wigner d-matrices. 
When written in a matrix form as~\eqref{eqn:Ul}, it is given by a product of three terms, where each term depends on one of Euler-angles.
In contrast to the recursive formulation written in terms of elements of a rotation matrix~\cite{IvaRueJPC96}, this structure is useful for a fast Fourier transform algorithm. 

In particular, when $l=1$, \eqref{eqn:Ul} results in
\begin{align*}
    U^1(R(\alpha,\beta,\gamma)) 
    & = 
    \begin{bmatrix}
        \cos\alpha & 0 & \sin \alpha\\
        0 & 1 & 0\\
        -\sin\alpha & 0 & \cos\alpha
    \end{bmatrix}
    \begin{bmatrix}
        1 & 0 & 0 \\
            0 & \cos\beta & -\sin\beta\\
            0 & \sin\beta & \cos\beta
        \end{bmatrix}
    \begin{bmatrix}
        \cos\gamma & 0 & \sin \gamma\\
        0 & 1 & 0\\
        -\sin\gamma & 0 & \cos\gamma
    \end{bmatrix}\\
    & = \begin{bmatrix} e_3 & e_1 & e_2 \end{bmatrix} R(\alpha,\beta,\gamma) 
    \begin{bmatrix} e_3 & e_1 & e_2 \end{bmatrix}^T.
\end{align*}
In other words, the real matrix representation of the order $l=1$ is similar to the rotation matrix itself. 

\subsection{Fourier Transform on $\SO$}

From Peter-Weyl theorem, real harmonics yield a complete, orthogonal basis on the square-integrable, real-valued functions on $\SO$. 
More explicitly, any real-valued $f\in\mathcal{L}^2(\SO)$ is expanded as
\begin{align}
    f(R(\alpha,\beta,\gamma)) &= \sum_{l=0}^\infty \sum_{m,n=-l}^l (2l+1) F^l_{m,n} U^l_{m,n}(\alpha,\beta,\gamma),\label{eqn:fR_IFT}
\end{align}
for real-valued Fourier parameters $F^l_{m,n}\in\Re$.
From \eqref{eqn:U_ortho}, the Fourier parameters are obtained by
\begin{equation}
    F^l_{m,n}= \pair{ U^l_{m,n}(R), f(R)}.\label{eqn:RFT}
\end{equation}

Next, we present a sampling theorem to evaluate \eqref{eqn:RFT} exactly with a finite number of samples. 
\begin{theorem}
    Consider a band-limited function represented by \eqref{eqn:fR_IFT} where $F^l_{m,n}=0$ for any $l\geq B$. 
    Define a uniform grid for $(\alpha,\beta,\gamma)\in[0,2\pi)\times [0,\pi]\times [0,2\pi)$ as
    \begin{equation}
        \alpha_j=\gamma_j = \frac{\pi}{B}j,\quad \beta_k = \frac{\pi(2k+1)}{4B},
    \end{equation}
    for $j,k\in\{0,\ldots, 2B-1\}$.
    The Fourier parameters for the band-limited function are given by
    \begin{equation}\label{eqn:Flmn}
        F^l_{m,n} = \sum_{j_1,j_2,k=0}^{2B-1} w_k U^l_{m,n}(R(\alpha_{j_1},\beta_k,\gamma_{j_2})) f(R(\alpha_{j_1},\beta_k,\gamma_{j_2})),
    \end{equation}
    where the weighting parameter $w_k\in\Re$ is defined as
    \begin{equation}\label{eqn:wk}
        w_k = \frac{1}{4B^3}\sin\beta_k \sum_{j=0}^{B-1} \frac{1}{2j+1}\sin((2j+1)\beta_k).
    \end{equation}
\end{theorem}
\begin{proof}
Define a sampling distribution as
\begin{equation}\label{eqn:s}
    s(R(\alpha,\beta,\gamma))=\sum_{j_1,k,j_2=0}^{2B-1} w_k \delta_{R(\alpha,\beta,\gamma),R(\alpha_{j_1},\beta_k,\gamma_{j_2})},
\end{equation}
which is the linear combination of grid points weighted by the parameter $w_k$. 
Since $\int_{\SO} f(R) \delta_{R,Q} dR = f(Q)$ for $Q\in\SO$, \eqref{eqn:RFT} yields the Fourier parameters for the sampling distribution as
\begin{equation*}
    S^l_{m,n} = \sum_{j_1,k,j_2=0}^{2B-1} w_k U^l_{m,n}(R(\alpha_{j_1},\beta_k,\gamma_{j_2})).
\end{equation*}
For the selected grid, it is straightforward to show $\sum_{j_1=0}^{2B-1} \sin m\alpha_{j_1} = 0$, $\sum_{j_1=0}^{2B-1} \cos m\alpha_{j_1} = 2B\delta_{m,0}$.
Therefore,  from \eqref{eqn:Psi},
\begin{equation*}
S^l_{m,n}  = 4B^2 \delta_{m,0}\delta_{n,0}\sum_{k=0}^{2B-1} w_k d^l_{0,0}(\beta_k).
\end{equation*}
In~\cite{DriHeaAAM94}, it is shown that the selected weight satisfies
\begin{equation*}
    \sum_{k=0}^{2B-1} w_k d^l_{0,0}(\beta_k) = \frac{1}{4B^2} \delta_{l,0},\quad l=0,\ldots 2B-1.
\end{equation*}
Therefore, the Fourier transform of the sampling distribution reduces to
\begin{equation*}
S^l_{m,n}=\delta_{m,0}\delta_{n,0}\delta_{l,0},\quad l=0,\ldots 2B-1.
\end{equation*}
Thus, the sampling distribution can be expanded as
\begin{equation}\label{eqn:s_FT}
s(R)=1+\sum_{l=2B}^\infty \sum_{m,n=-l}^l S^l_{m,n} U^l_{m,n}(R).
\end{equation}

Next, define $g(R)=f(R) s(R)$. 
From \eqref{eqn:s} and using the property of the delta function, it is straightforward to show that the Fourier parameters for $g(R)$ is given by
\begin{equation*}
    G^l_{m,n} = \sum_{j_1,k,j_2=0}^{2B-1} w_k f(R(\alpha_{j_1},\beta_k,\gamma_{j_2})) 
{U^l_{m,n}(R(\alpha_{j_1},\beta_k,\gamma_{j_2}))}.
\end{equation*}
On the other hand, using  \eqref{eqn:s_FT}, 
\begin{align*}
    g(R)&=f(R)+f(R)\sum_{l=2B}^\infty \sum_{m,n=-l}^l S^l_{m,n} U^l_{m,n}(R).
\end{align*}
Now, we show that the above expression for $g(R)$ yields the Fourier parameters that are identical to those of $f(R)$.
Since $f(R)$ can be expanded as a linear combination of $U^{l_1}$ for $0\leq l_1 \leq B-1$. 
The last term of the above equation is expanded by the product $U^{l_1}U^{l_2}$ with $0\leq l_1 \leq B-1$  and $2B\leq l_2$. 
According to the Clebsch-Gordon theorem, $U^{l_1}U^{l_2}$ is a linear combination of $U^{l_3}$ for $|l_1-l_2|\leq l_3 \leq l_1+l_2$. 
We have $\min|l_1-l_2|=2B-1-B=B+1$. 
As such, $f(R)$ and $g(R)$ share the Fourier coefficients in the given band limit, i.e., $F^l_{m,n}=G^l_{m,n}$ for $l\in\{0,\ldots B-1\}$.
This shows \eqref{eqn:Flmn}.
\hfill $\blacksquare$
\end{proof}

Therefore, Fourier parameters of any band-limited function with the bandwidth $B$ can be computed exactly with $(2B)^3$ samples evaluated at the given grid points. 
Utilizing this, we present a fast Fourier transform. 

\subsection{Fast Fourier Transform on $\SO$}\label{sec:FFT}

Substituting \eqref{eqn:U_mn} into \eqref{eqn:Flmn}, the Fourier transform represented by \eqref{eqn:Flmn} can be executed in the following sequence. 
For each $k\in\{0,\ldots  2B-1\}$, let $\mathbf{F}^{k},\mathbf{G}^k\in\Re^{(2B-1)\times(2B-1)}$ be
\begin{align*}
    \mathbf{F}^k_{m,n}  = \sum_{j_1,j_2=0}^{2B-1}  f(R(\alpha_{j_1}, \beta_k, \gamma_{j_2}))\sin(m\alpha_{j_1}+n\gamma_{j_2}),\\
    \mathbf{G}^k_{m,n}  = \sum_{j_1,j_2=0}^{2B-1}  f(R(\alpha_{j_1}, \beta_k, \gamma_{j_2}))\cos(m\alpha_{j_1}+n\gamma_{j_2}),
\end{align*}
which can be computed by a real-valued fast Fourier transform algorithm developed in $\Re$, such as~\cite{SorJonTAS87}. 
The Fourier parameters defined in~\eqref{eqn:U_mn} are evaluated by
\begin{align*}
    & F^l_{m,n} = \\
    &
    \begin{cases}
        & \sum_{k=0}^{2B-1} w_k \{ -\frac{1}{2}(\mathbf{G}^k_{m,-n}-\mathbf{G}^k_{m,n})\Psi^l_{-m,n}(\beta_k) + \frac{1}{2}(\mathbf{G}^k_{m,-n}+\mathbf{G}^k_{m,n})\Psi^l_{m,n}(\beta_k)\}\\
        & \hfill (m\geq 0,n\geq 0) \text{ or } (m<0,n<0),\\
        & \sum_{k=0}^{2B-1} w_k \{ -\frac{1}{2}(\mathbf{F}^k_{m,n}+\mathbf{F}^k_{m,-n})\Psi^l_{-m,n}(\beta_k) + \frac{1}{2}(\mathbf{F}^k_{m,n}-\mathbf{F}^k_{m,-n})\Psi^l_{m,n}(\beta_k)\}\\
        & \hfill (m\geq 0,n< 0) \text{ or } (m<0,n\geq 0).
    \end{cases}
\end{align*}

\subsection{Derivatives of Representation}

Similar with~\eqref{eqn:ul}, the derivatives of $U^l(R)$ at $R=I_{3\times 3}$ results in the real matrix representation of $\so$. Here we use the same notation as the complex valued case as
\begin{equation}
    u^l(\eta)= \frac{d}{d\epsilon}\bigg|_{\epsilon=0} U^l(\exp(\epsilon\hat\eta)),\label{eqn:ul_real}
\end{equation}
for $\eta\in\Re^3$. 
From the linearity, $u^l(\eta)$ for any $\eta$ can be evaluated by the results of the following theorem. 

\begin{theorem}
    The derivatives of $U^l(R)$ introduced at~\eqref{eqn:ul_real} are given as follows for $\eta\in\{e_1,e_2,e_3\}$.
    \begin{align}
        u^l_{m,n}(e_1) & = \begin{cases}
            \frac{1}{2}(m+n)\sqrt{(l+|m|)(l-|m|+1)} \\
            \hfill (m\geq 2, n=-m+1) \text{ or } (m\leq -2, n=-m-1), \\
            -\frac{1}{2}(m+n)\sqrt{(l-|m|)(l+|m|+1)} \\
            \hfill (1\leq m \leq l-1, n=-m-1) \text{ or } (-l+1\leq m \leq -1, n=-m+1), \\
            \frac{1}{\sqrt{2}} \sqrt{l(l+1)} \hfill  m=-1, n=0,\\
            -\frac{1}{\sqrt{2}} \sqrt{l(l+1)} \hfill m=0, n=-1,\\
            0 \hfill \text{otherwise},
        \end{cases}\label{eqn:ul1}\\
        u^l_{m,n}(e_2) & = \begin{cases}
            \frac{1}{2} \sqrt{(l+|m|)(l-|m|+1)} \\
            \hfill (m\geq 2, n=m-1) \text{ or } (m\leq -2, n=m+1), \\
            -\frac{1}{2} \sqrt{(l-|m|)(l+|m|+1)} \\
            \hfill (1\leq m\leq l-1, n=m+1) \text{ or } (-l+1\leq m\leq -1, n=m-1), \\
            \frac{1}{\sqrt{2}} \sqrt{l(l+1)} 
            \hfill m=1, n=0,\\
            -\frac{1}{\sqrt{2}} \sqrt{l(l+1)} 
            \hfill m=0, n=1,\\
            0 \hfill \text{otherwise},
        \end{cases},\label{eqn:ul2}\\
        u^l_{m,n}(e_3) & =
        \begin{cases}
            -m  & (m=-n) \text{ and } (m \neq 0)\\
            0 & \text{otherwise}
        \end{cases}\label{eqn:ul3}
    \end{align}
\end{theorem}
\begin{proof}
    Since $\exp(\epsilon\hat e_3) = R(\epsilon,0,0)$, from \eqref{eqn:U_mn}, 
\begin{align}
    u^l_{m,n}(e_3) & =
    \begin{cases}
        \frac{d}{d\epsilon}\big|_{\epsilon=0} \cos m\epsilon \Psi^l_{m,n}(0) & (m\geq 0, n\geq 0) \text{ or } (m<0,n<0)\\
        \frac{d}{d\epsilon}\big|_{\epsilon=0} -\sin m\epsilon \Psi^l_{-m,n} (0) & (m\geq 0, n < 0) \text{ or } (m<0,n\geq0)
    \end{cases},\nonumber\\
    & =
    \begin{cases}
        0 & (m\geq 0, n\geq 0) \text{ or } (m<0,n<0)\\
        -m \Psi^l_{-m,n} (0) & (m\geq 0, n < 0) \text{ or } (m<0,n\geq0)
    \end{cases}\label{eqn:u_lmn_3}
\end{align}
According to \eqref{eqn:d_properties}, $d^l_{m,n}(0)=\delta_{m,n}$. Therefore, \eqref{eqn:Psi} yields
    \begin{align*}
        \Psi^l_{m,n}(0) & = 
        \begin{cases}
            (-1)^{m-n} \delta_{|m|,|n|} & mn\neq 0,\\
            0 & m = 0 \text{ xor } n=0,\\
            1 & m=n=0,
        \end{cases}
    \end{align*}
    Substituting these into \eqref{eqn:u_lmn_3}, 
\begin{equation*}
    u^l_{m,n}(e_3) =
    \begin{cases}
        0 & (m\geq 0, n\geq 0) \text{ or } (m \leq 0, n \leq 0)\\
        -m(-1)^{m-n}\delta_{|m|,|n|} & (m > 0, n < 0) \text{ or } (m < 0, n >0),
    \end{cases}
\end{equation*}
which reduces to \eqref{eqn:ul3}. 
The other can be shown similarly, using $\exp(\epsilon\hat e_2) = R(0,\epsilon,0)$, and $u^l(e_1)=u^l(e_2\times e_3)=[u^l(e_2), u^l(e_3)]$.
\hfill $\blacksquare$
\end{proof}

\subsection{Clebsch-Gordon Coefficients}

Here we find the Clebsch-Gordon coefficients for real harmonics. 
The objective is to write a product of two real harminics as a linear combination of other harmonics.
Repeatedly using the property of Kronecker product, namely $(AC)\otimes (BD) = (A\otimes B)(C \otimes D)$ for arbitrary compatible matrices $A,B,C,D$, \eqref{eqn:Ul_T} results in
\begin{equation*}
    U^{l_1}(R) \otimes U^{l_2}(R) = (\overline{T}^{l_1} \otimes \overline{T}^{l_2}) (D^{l_1}(R) \otimes D^{l_2}(R)) (T^{l_1} \otimes   T^{l_2})^T .
\end{equation*}
Substituting \eqref{eqn:Clebsch_Gordon}, this is rearranged into
\begin{equation}
    U^{l_1}(R) \otimes U^{l_2}(R) = c_{l_1,l_2} \bracket{ \bigoplus_{l=|l_1-l_2|}^{l_1+l_2} U^l(R) } \overline{c}_{l_1,l_2}^T,
\end{equation}
where the Clebsch-Gordon matrix for the real harmonics, namely $c_{l_1,l_2}\in \Cp^{(2l_1+1)(2l_2+1)\times(2l_1+1)(2l_2+1)}$ is defined as
\begin{equation}\label{eqn:cb_real}
    c_{l_1,l_2} = (\overline{T}^{l_1} \otimes \overline{T}^{l_2}) C_{l_1,l_2} \bracket{ \bigoplus_{l=|l_1-l_2|}^{l_1+l_2} (T^l)^T }.
\end{equation}
Interestingly, while the Clebsch-Gordon coefficients $C_{l_1,l_2}$ for the complex harmonics are real-valued, those for real harmonics can be complex-valued as the matrix $T^l$ is composed of real or imaginary elements. 
In the element-wise form,
\begin{equation}
    U^{l_1}_{m_1,n_1} (R) U^{l_2}_{m_2,n_2}(R) = \sum_{l=|l_1-l_2|}^{l_1+l_2} \sum_{m,n=-l}^l c^{l,m}_{l_1,m_1,l_2,m_2} \overline{c}^{l,n}_{l_1,n_1,l_2,n_2} U^{l}_{m,n}(R).\label{eqn:Ul1_Ul2}
\end{equation}

The following theorem presents two properties of the real Clebsch-Gordon coefficients, and provides an alternative, simpler method to evaluate~\eqref{eqn:cb_real}.

\begin{theorem}
    The Clebsch-Gordon coefficients defined in \eqref{eqn:cb_real} is unitary, i.e., 
    \begin{equation}
        c_{l_1,l_2} \overline{c}_{l_1,l_2}^T = I. \label{eqn:c_unitary}
    \end{equation}
    They have zero values for the following cases,
\begin{equation}
    c^{l,m}_{l_1,m_1,l_2,m_2} = 0, \text{ if  $|m|\neq |m_1+m_2|$ or $|m|\neq |m_1-m_2|$}.\label{eqn:c_lm_zero}
\end{equation}
The remaining non-zero values for $m\in\{m_1+m_2,-m_1-m_2,m_1-m_2,-m_1+m_2\}$ are evaluated according to Table \ref{tab:cb_real}.
\end{theorem}
\begin{proof}
    Equation \eqref{eqn:c_unitary} follows from the fact that $C_{l_1,l_2}$ and $T^l$ are unitary as shown at \eqref{eqn:C_ortho_0}, \eqref{eqn:C_ortho_1}, and \eqref{eqn:T_unitary}.
    Next, following the ordering rules \eqref{eqn:order_scheme_c} and \eqref{eqn:order_scheme_r}, \eqref{eqn:cb_real} is rewritten into an element-wise form as
    \begin{equation*}
        c^{l,m}_{l_1,m_1,l_2,m_2} = \sum_{p_1=-l_1}^{l_1} \sum_{p_2 = -l_2}^{l_2} \sum_{p=-l}^l \overline T^{l_1}_{m_1,p_1} \overline T^{l_2}_{m_2,p_2} T^l_{m,p} C^{l,p}_{l_1,p_1,l_2,p_2}.
    \end{equation*}
    From \eqref{eqn:Tmn}, we have $T^l_{m,n}=0$ for $|m|\neq|n|$. 
    Also $C^{l,m}_{l_1,m_1,l_2,m_2}=0$ if $m\neq m_1+m_2$.
    Using these, the triple summation in the above equation reduces to
    \begin{equation*}
        c^{l,m}_{l_1,m_1,l_2,m_2}  = \sum_{p_1\in\{-m_1,m_1\}} \sum_{p_2 =\{ -m_2,m_2\}} \delta_{|m|,|p_1+p_2|} \overline T^{l_1}_{m_1,p_1} \overline T^{l_2}_{m_2,p_2} T^l_{m,p_1+p_2} C^{l,p_1+p_2}_{l_1,p_1,l_2,p_2},
    \end{equation*}
    which follows \eqref{eqn:c_lm_zero}.

    Suppose $m_1,m_2>0$. 
    Equation \eqref{eqn:c_lm_zero} implies $c^{l,m}_{l_1,m_1,l_2,m_2}\neq 0$ when $m=m_1+m_2$ or $m=-m_1-m_2$. 
    For the former, using \eqref{eqn:Tmn} and \eqref{eqn:C_sym},
    \begin{align*}
        c^{l,m_1+m_2}_{l_1,m_1,l_2,m_2} & =  \overline T^{l_1}_{m_1,m_1} \overline T^{l_2}_{m_2,m_2} T^l_{m_1+m_2,m_1+m_2} C^{l,m_1+m_2}_{l_1,m_1,l_2,m_2}\\
                                        & \quad + \overline T^{l_1}_{m_1,-m_1} \overline T^{l_2}_{m_2,-m_2} T^l_{m_1+m_2,-m_1-m_2} C^{l,-m_1-m_2}_{l_1,-m_1,l_2,-m_2} ,\\
                                        & = \frac{1}{\sqrt{8}} ( 1 + (-1)^{l_1+l_2-l}) C^{l,m_1+m_2}_{l_1,m_1,l_2,m_2}.
    \end{align*}
    For the latter, 
    \begin{align*}
        c^{l,-m_1-m_2}_{l_1,m_1,l_2,m_2} & = \overline T^{l_1}_{m_1,-m_1} \overline T^{l_2}_{m_2,-m_2} T^l_{-m_1-m_2,-m_1-m_2} C^{l,-m_1-m_2}_{l_1,-m_1,l_2,-m_2} \nonumber\\
                                         &\quad + \overline T^{l_1}_{m_1,m_1} \overline T^{l_2}_{m_2,m_2} T^l_{-m_1-m_2,m_1+m_2} C^{l,m_1+m_2}_{l_1,m_1,l_2,m_2} ,\\ 
                                         & = \frac{i}{\sqrt{8}} ( 1 - (-1)^{l_1+l_2-l}) C^{l,-m_1-m_2}_{l_1,-m_1,l_2,-m_2}.
    \end{align*}
    These show the shaded cells of Table \ref{tab:cb_real}. 
    Other parts of the table can be shown similarly.\hfill $\blacksquare$ 
\end{proof}

\newcolumntype{C}{>{$\displaystyle} Sc <{$}}
\setlength{\cellspacetoplimit}{5pt}
\setlength{\cellspacebottomlimit}{5pt}

\begin{table}
    \caption{Clebsch-Gordon coefficients for real harmonics}\label{tab:cb_real}
    \begin{center}
        \begin{threeparttable}
        \scriptsize\selectfont
            \begin{tabular}{ CCCC | C|C|C|C}\toprule
                m_1 & m_2 & \shortstack[c]{$m_1+$\\$m_2$} & \shortstack[c]{$m_1-$\\$m_2$} & \dfrac{c^{l,m_1+m_2}_{l_1,m_1,l_2,m_2}}{ C^{m_1+m_2}_{l_1,m_1,l_2,m_2}}  & \dfrac{c^{l,-m_1-m_2}_{l_1,m_1,l_2,m_2}}{ C^{m_1+m_2}_{l_1,m_1,l_2,m_2}}  & \dfrac{c^{l,m_1-m_2}_{l_1,m_1,l_2,m_2}}{ C^{m_1-m_2}_{l_1,m_1,l_2,-m_2}}  & \dfrac{c^{l,-m_1+m_2}_{l_1,m_1,l_2,m_2}}{ C^{m_1-m_2}_{l_1,m_1,l_2,-m_2}}\\
                      & & & & \multicolumn{2}{c|}{$(|m_1+m_2| \leq l \leq l_1+l_2)$} & \multicolumn{2}{c}{$(|m_1-m_2| \leq l \leq l_1+l_2)$} \\ \midrule
                 0 & 0 & 0 & 0 & 1 & 1 & 1 & 1\\\midrule
                 + & 0 & + & + & \frac{1}{2}\eta & \frac{i}{2}\zeta & \frac{1}{2}\eta & \frac{i}{2}\zeta \\\midrule
                 + & + & + & + & \cellcolor[HTML]{F0F0F0}\frac{1}{\sqrt{8}}\eta & \cellcolor[HTML]{F0F0F0}\frac{i}{\sqrt{8}}\zeta & \frac{(-1)^{m_2}}{\sqrt{8}} \eta & \frac{i(-1)^{m_2}}{\sqrt{8}}\zeta \\\hline
                 + & + & + & 0 & \cellcolor[HTML]{F0F0F0}\frac{1}{\sqrt{8}}\eta & \cellcolor[HTML]{F0F0F0}\frac{i}{\sqrt{8}}\zeta & \frac{(-1)^{m_1}}{2}\eta & \frac{(-1)^{m_1}}{2}\eta \\\hline
                 + & + & + & - & \cellcolor[HTML]{F0F0F0}\frac{1}{\sqrt{8}}\eta & \cellcolor[HTML]{F0F0F0}\frac{i}{\sqrt{8}}\zeta & -\frac{i(-1)^{m_1}}{\sqrt{8}}\zeta & \frac{(-1)^{m_1}}{\sqrt{8}} \\\midrule
                 0 & + & + & - & \frac{1}{2}\eta & \frac{i}{2}\zeta & \frac{1}{2}\eta & \frac{i}{2}\zeta \\\midrule
                 - & + & + & - & \frac{i(-1)^{m_1}}{\sqrt{8}}\zeta & -\frac{(-1)^{m_1}}{\sqrt{8}}\eta  & \frac{1}{\sqrt{8}}\eta & \frac{i}{\sqrt{8}}\zeta \\\hline
                 - & + & 0 & - & \frac{i(-1)^{m_1}}{2}\zeta & \frac{i(-1)^{m_1}}{2}\zeta & \frac{1}{\sqrt{8}}\eta & \frac{i}{\sqrt{8}}\zeta\\\hline
                 - & + & - & - & \frac{(-1)^{m_2}}{\sqrt{8}}\eta & \frac{i(-1)^{m_2}}{\sqrt{8}}\zeta &\frac{1}{\sqrt{8}}\eta & \frac{i}{\sqrt{8}}\zeta \\\midrule
                 - & 0 & - & - & \frac{1}{2}\eta & \frac{i}{2}\zeta & \frac{1}{2}\eta & \frac{i}{2}\zeta \\\midrule
                 - & - & - & - & \frac{i}{\sqrt{8}}\zeta  & -\dfrac{1}{\sqrt{8}}\eta & -\frac{i(-1)^{m_2}}{\sqrt{8}}\zeta & \frac{(-1)^{m_2}}{\sqrt{8}}\eta \\\hline
                 - & - & - & 0 & \frac{i}{\sqrt{8}}\zeta  & -\dfrac{1}{\sqrt{8}}\eta & \frac{(-1)^{m_1}}{2}\eta & \frac{(-1)^{m_1}}{2}\eta \\\hline
                 - & - & - & + & \frac{i}{\sqrt{8}}\zeta  & -\dfrac{1}{\sqrt{8}}\eta & \frac{(-1)^{m_1}}{\sqrt{8}}\eta & \frac{i(-1)^{m_1}}{\sqrt{8}} \\\midrule
                 0 & - & - & + & \frac{1}{2}\eta & \frac{i}{2}\zeta & \frac{1}{2}\eta & \frac{i}{2}\zeta \\\midrule
                 + & - & - & + & \frac{(-1)^{m_1}}{\sqrt{8}}\eta & \frac{i(-1)^{m_1}}{\sqrt{8}}\zeta & -\frac{i}{\sqrt{8}}\zeta & \frac{1}{\sqrt{8}}\eta \\\hline
                 + & - & 0 & + & \frac{i(-1)^{m_1}}{2}\zeta & \frac{i (-1)^{m_1}}{2}\zeta &  -\frac{i}{\sqrt{8}}\zeta & \frac{1}{\sqrt{8}}\eta \\\hline
             + & - & + & + & \frac{i(-1)^{m_2}}{\sqrt{8}}\zeta & -\frac{(-1)^{m_2}}{\sqrt{8}}\eta  & -\frac{i}{\sqrt{8}}\zeta & \frac{1}{\sqrt{8}}\eta \\\bottomrule 
            \end{tabular}
            \begin{tablenotes}
            \item $\eta=(-1)^{l_1+l_2-l}+1$,\quad $\zeta=(-1)^{l_1+l_2-l}-1$.
            \end{tablenotes}
        \end{threeparttable}
    \end{center}
\end{table}

From \eqref{eqn:c_lm_zero}, the summation at \eqref{eqn:Ul1_Ul2} reduces to
\begin{equation}
    U^{l_1}_{m_1,n_1} (R) U^{l_2}_{m_2,n_2}(R) =  \sum_{m\in M} \sum_{n\in N} \sum_{l=\max\{|l_1-l_2|,|m|,|n|\}}^{l_1+l_2} c^{l,m}_{l_1,m_1,l_2,m_2} \overline{c}^{l,n}_{l_1,n_1,l_2,n_2} U^{l}_{m,n}(R).\label{eqn:Ul1_Ul2_reduced}
\end{equation}
where
\begin{align*}
    M& =\{m_1+m_2,m_1-m_2,-m_1+m_2,-m_1-m_2\},\\
    N&=\{n_1+n_2,n_1-n_2,-n_1+n_2,-n_1-n_2\}.
\end{align*}
As stated above, the Clebsch-Gordon coefficients for real harmonics is complex in general.
However, by investigating Table \ref{tab:cb_real}, one can show that the product $ c^{l,m}_{l_1,m_1,l_2,m_2} \overline{c}^{l,n}_{l_1,n_1,l_2,n_2}$ in~\eqref{eqn:Ul1_Ul2_reduced} is real always.
This is expected as $U^l_{m,n}(R)$ is real-valued always.

\section{Software Implementation}

The presented real harmonic analysis on $\SO$ has been implemented by C++, and the resulting software package has been disclosed as an Open Source Library  at \url{https://github.com/fdcl-gwu/FFTSO3}.
This provides the following functionality and features:
\begin{itemize}
    \item Complex and real harmonic analysis on $\SO$ 
    \item Complex and real harmonic analysis on $\Sph^2$
    \item Fast Fourier transform based on the EigenFFT library~\cite{eigenweb}
    \item Clebsch-Gordon coefficients / derivatives of harmonics
    \item Multi-core parallel computation with the OpenMP library~\cite{DagMenICSE98}
\end{itemize}

\subsection{Validation}

The software package is validated by software unit-testing implemented via Google Test~\cite{googletest}.
There are several tests to verify that numerical results are consistent with theoretical analysis.
Here we show the results of a particular unit-testing designed to ensure that the composition of the inverse transform and the forward transform yields the identity map on the space of Fourier parameters~\cite{KosRocJFAA08}.

Specifically, we define a function $f(R):\SO\rightarrow \Re$ as the inverse Fourier transform presented in~\eqref{eqn:fR_IFT}. 
We assume that the function is band-limited so that $F^l_{m,n}=0$ any $l\geq B$.
Next, each component of Fourier parameter $F^l_{m,n}$ is randomized by a uniform-distribution in $[-1,1]$.
The corresponding function $f(R)$ is transformed via the fast Fourier transform  described in Section~\ref{sec:FFT}, to compute a new set of Fourier parameters, namely $G^l_{m,n}$. 
Theoretically, $G^l_{m,n}$ should be identical to the original Fourier parameters $F^l_{m,n}$ of $f(R)$, as the composition of the inverse transform and the forward transform is the identity map on the space of Fourier parameters.
Numerical errors of the presented software library are measured by the discrepancy between $F^l_{m,n}$ and $G^l_{m,n}$ calculated as the sum of the matrix norm,
\[
    \mathrm{error} = \sum_{l=0}^{B-1} \| F^l -G^l \|.
\]
Table \ref{tab:err} summarizes the above error for varying band-limits. 
The error increases as the band-limit, since the error measure is defined as the sum over the index $l$. 
However, all of the errors are under an acceptable level relative to the machine precision of double-type variables that the presented software uses. 

\begin{table}[h]
    \caption{Numerical error for the composition of the forward and inverse Fourier transforms}\label{tab:err}
    \begin{center}
        \begin{tabular}{ccccc}\toprule
            $B$ & 8 & 16 & 32 & 64 \\\midrule
            $\mathrm{error}$ & $7.2528\times 10^{-14}$ & $5.8972\times 10^{-13}$ & $4.8600\times 10^{-12}$ & $4.0484\times 10^{-11}$\\\bottomrule
        \end{tabular}
    \end{center}
\end{table}

\subsection{Benchmark}

Next, we check the computational time required to perform Fourier transforms. 
Fast forward Fourier transform is executed for the trace function on $\SO$, with varying bandwidth $B\in\{8,16,32,64,128\}$ and number of threads $N_{thread}\in\{1,2,4,8\}$.
The resulting computation time is measured with Intel i7-6800K 3.4GHz CPU, and the test is repeated 10 times to compute the average computation time.

Figure \ref{fig:fwd} illustrates the computation time with respect to the band-with for several number of threads. 
The computation time increases as $B$ is increased, but it reduces as more threads are available. 
More specifically, the next subfigure shows the speed-up factor, the ratio of the computation time with a single thread to that of multiple threads. 
Ideally, the speed-up factor is identical to the number of threads, as indicated by the dotted line. 
However, due to the overhead in distributing the tasks in multiple threads and collecting results, as well as delay in accessing a shared memory, 
the speed-up factor is often lower than the number of threads in practice. 
It is illustrated that the speed-up factor increases as $B$ is increased, i.e., parallel processing is more beneficial if the bandwidth is greater. 

Figure \ref{fig:inv} is for the inverse transform. 
Compared with the forward transform, only a fraction of time is required. 
As such, the speed-up factor is relatively low even for higher bandwidth. 
At least, they are greater than one, indicating that the computation time is reduced by parallel processing. 

\begin{figure}
    \centerline{
        \subfigure[CPU time (sec)]{
            \includegraphics[height=0.38\textwidth]{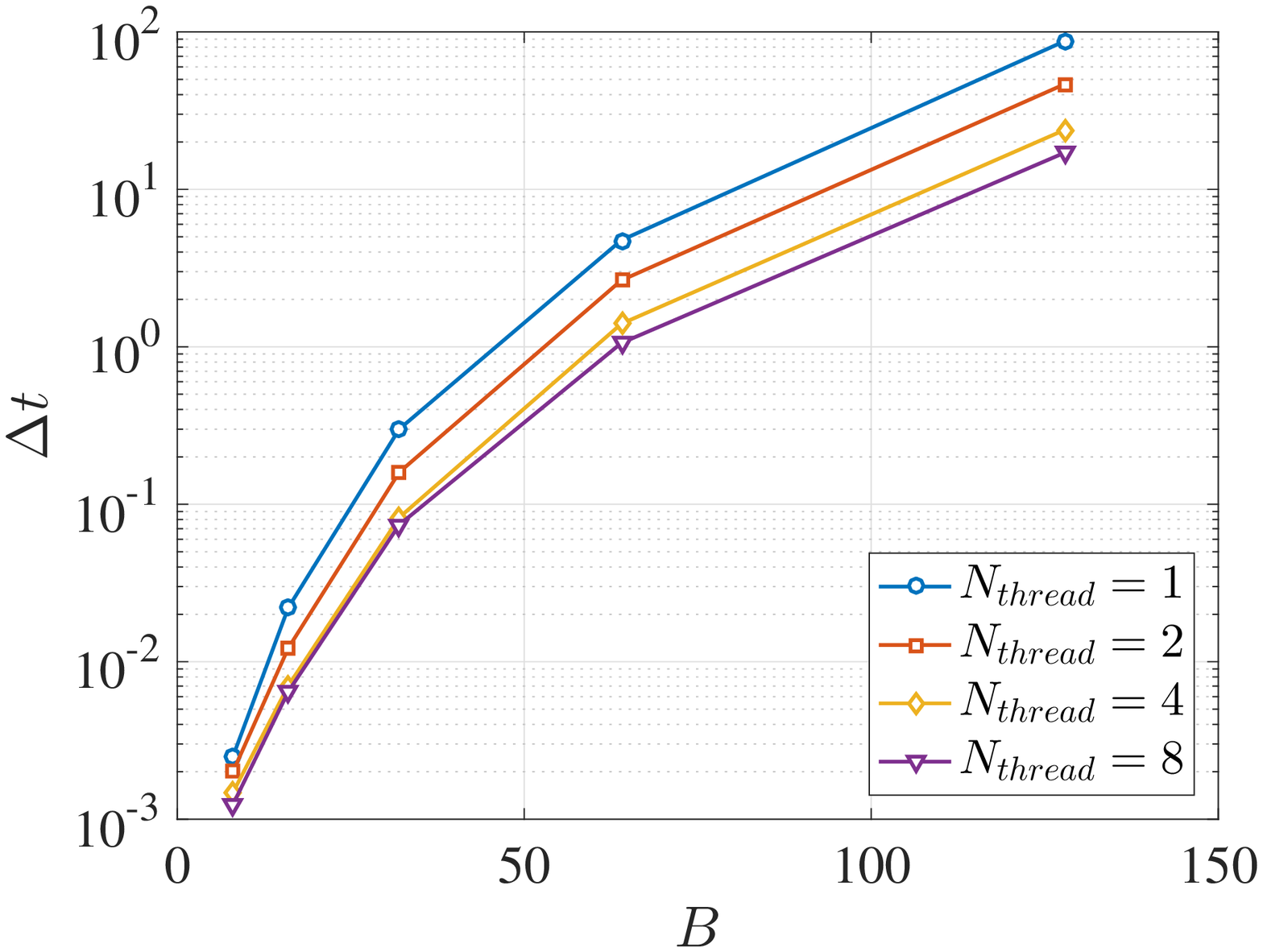}
        }
        \hfill
        \subfigure[Speed-up factor]{
            \includegraphics[height=0.38\textwidth]{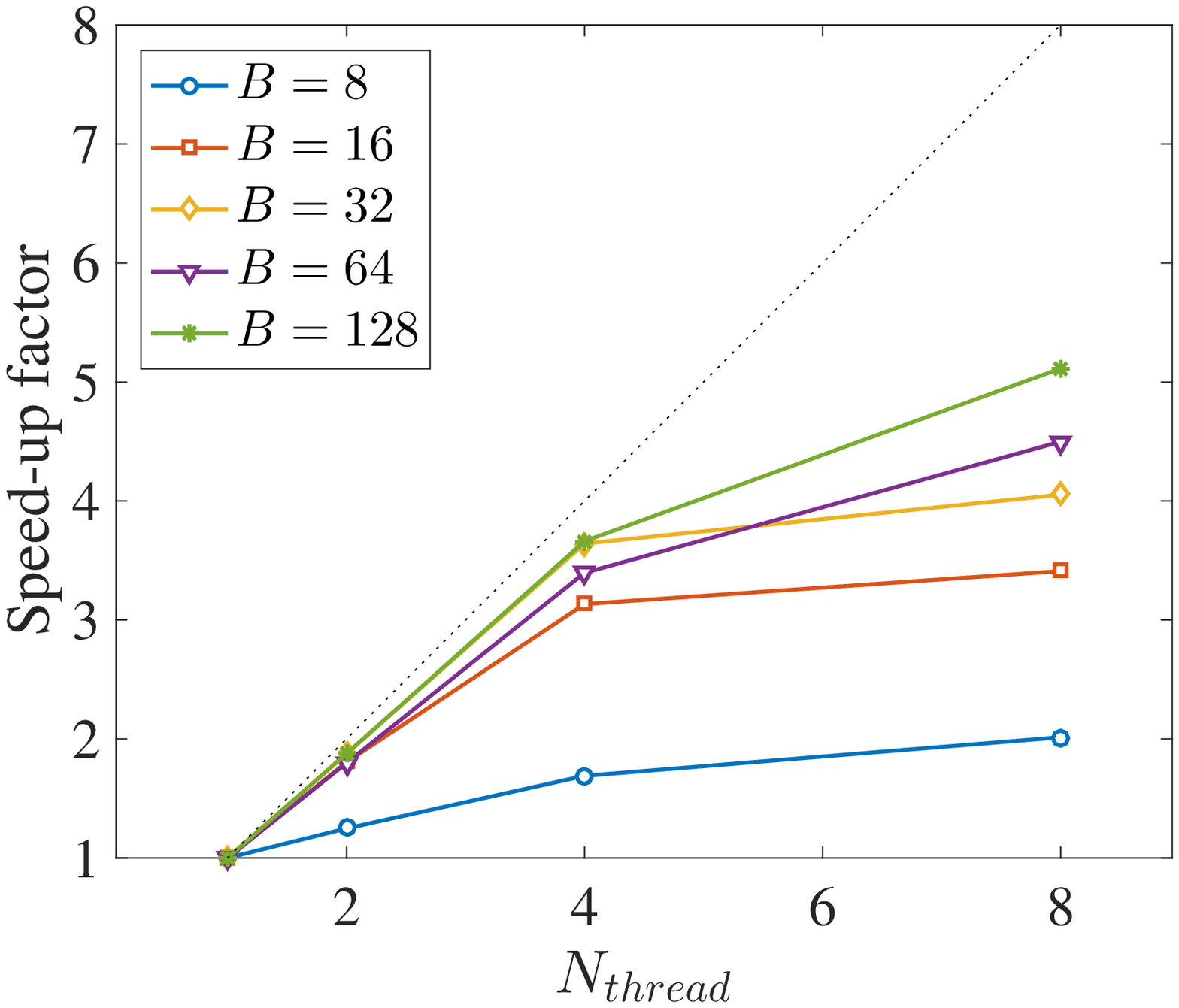}
        }
    }
    \caption{Benchmark results for forward Fourier transform}\label{fig:fwd}
\end{figure}
\begin{figure}
    \centerline{
        \subfigure[CPU time (sec)]{
            \includegraphics[height=0.38\textwidth]{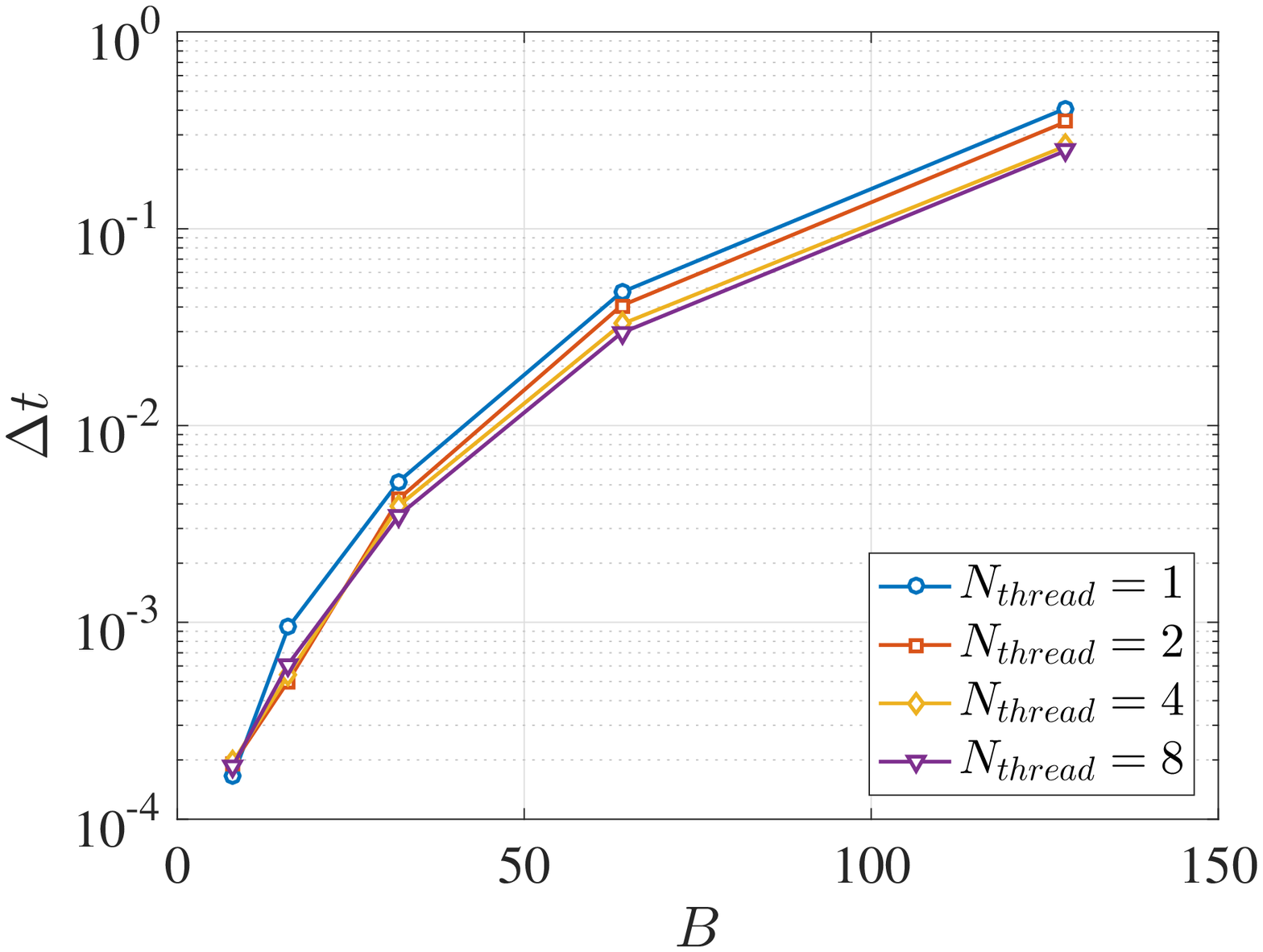}
        }
        \hfill
        \subfigure[Speed-up factor]{
            \includegraphics[height=0.38\textwidth]{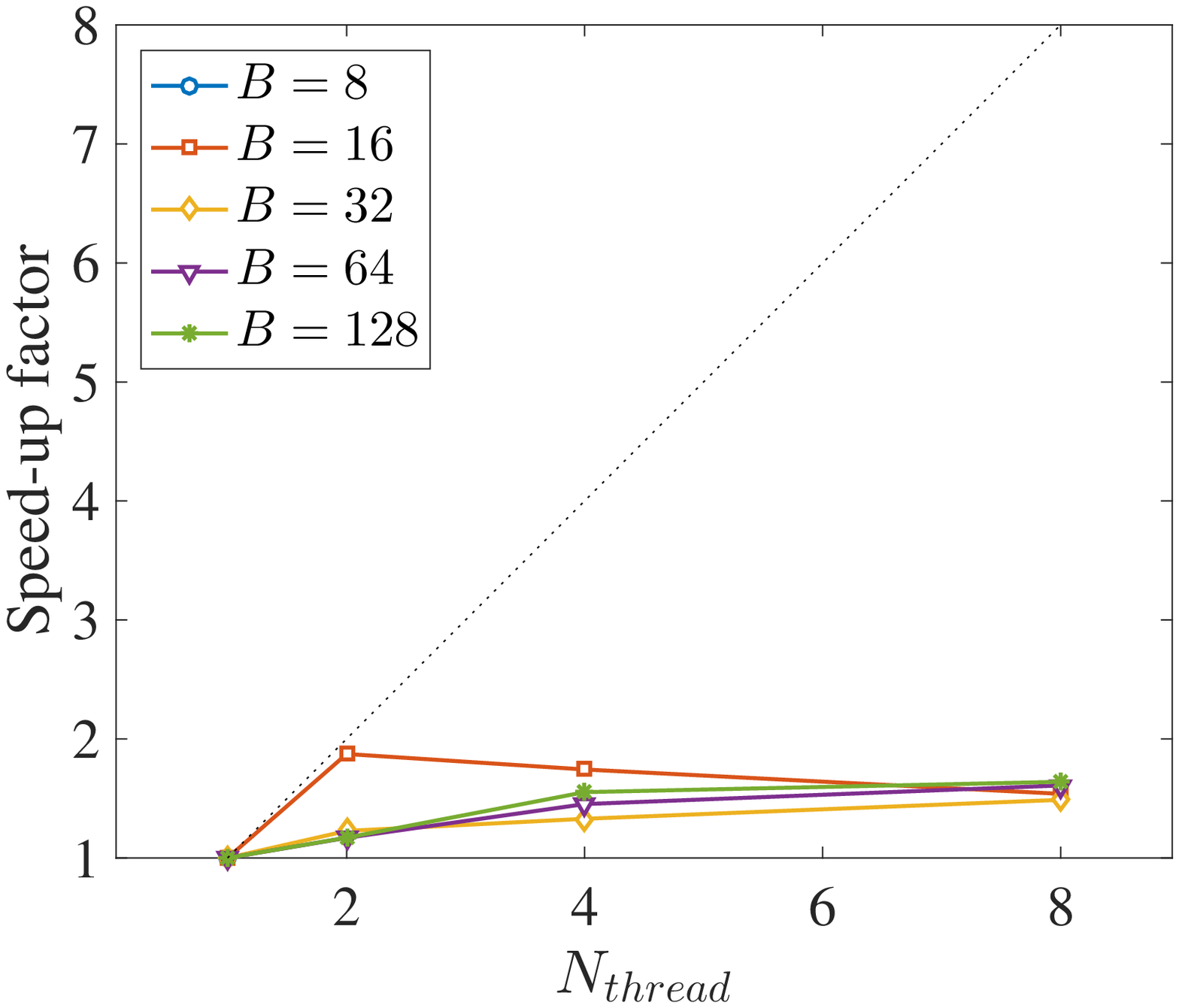}
        }
    }
    \caption{Benchmark results for inverse Fourier transform}\label{fig:inv}
\end{figure}

\section{Application}

\subsection{Spherical Shape Matching}

We apply the above software package for spherical shape correlation and matching. 
Consider an object embedded in $\Re^3$.
Assuming that it has zero genus, i.e., no hole, it's shape can be completely described by deforming a sphere. 
Let $f(x):\Sph^2\rightarrow\Re$ be the deformed radius along the radial direction specified by the unit-vector $x\in\Sph^2$.
Suppose the object is rotated by a rotation described by $R\in\SO$, and let $g(x):\Sph^2\rightarrow\Re$ describe the rotated object. 
The spherical shape matching problem considered in this paper is to find the rotation matrix $R\in\SO$ for given shape functions $f(x)$ and $g(x)$. 

This can be formulated as the following optimization problem on $\SO$.
For a rotation matrix $R$, let its cost be the discrepancy between $g(x)$ and $f(x)$ rotated by $R$, i.e., 
\begin{align*}
    \mathcal{J}(R) & = \frac{1}{2}\| g(x) - f(R^T x) \|^2 \\
                   & = \frac{1}{2} \|g(x)\|^2 + \frac{1}{2}\|f(x)\|^2 -
    \pair{ g(x), f(R^T x) }_{\mathcal{L}(\Sph^2)},
\end{align*}
where we have used the fact $\|f(R^Tx)\|=\|f (x)\|$. 

The true rotation matrix $R$ minimizes the cost function $\mathcal{J}(R)$, or equivalently, it maximizes the measure of correlation defined by $\mathcal{C}(R) = \pair{ g(x), f(R^T x) }_{\mathcal{L}(\Sph^2)}$, i.e., 
\[
    R = \operatorname*{arg\,max} \; \{\mathcal{C}(R)\}.
\]
However, evaluating the above inner product requires an integration over $\Sph^2$, and it may be cumbersome to repeat it at every iteration of numerical optimization. 

Instead, the cost function can be evaluated without any integration, by utilizing the fundamental property of representation given by \eqref{eqn:DRf}.
Let the Fourier transform of $f(x)$ and $g(x)$ with a bandwidth $B$ be
\begin{align*}
    f(x) = \sum_{l=0}^{B-1} (F^l)^T S^l(x),\\
    g(x) = \sum_{l=0}^{B-1} (G^l)^T S^l(x),
\end{align*}
for Fourier parameters $F^l,G^l\in\Re^{(2l+1)\times 1}$. 
From \eqref{eqn:SRx},
\begin{equation*}
    f(R^T x) = \sum_{l=0}^{B-1} (F^l)^T S^l(R^T x) = \sum_{l=0}^{B-1} (U^l(R)F^l)^T S^l(x).
\end{equation*}
The orthogonality of real spherical harmonics implies $\pair{S^{l_1}(x), (S^{l_2}(x))^T }_{\mathcal{L}(\Sph^2)} = \frac{1}{4\pi}\delta_{l_1,l_2}I_{(2l_1+1)\times (2l_1+1)}$. 
Therefore, the correlation function reduces to
\begin{equation}
    \mathcal{C}(R) = \frac{1}{4\pi}\sum_{l=0}^{B-1} (G^l)^T U^l(R) F^l,\label{eqn:cor}
\end{equation}
which is an algebraic equation of Fourier parameters that does not require any integration. 

Furthermore, its gradient can be evaluated by \eqref{eqn:ul1}--\eqref{eqn:ul3} as follow. 
For $\eta = (\eta_1,\eta_2,\eta_3)\in\Re^3$, a directional derivatives of the correlation function is given by
\begin{equation*}
    \frac{d}{d\epsilon} \bigg|_{\epsilon=0} \mathcal{C}(R \exp(\epsilon\hat\eta)))
    =\nabla \mathcal{C}(R)  \cdot \eta,
\end{equation*}
where the gradient is considered as a $3\times 1$ vector, i.e., $\nabla \mathcal{C}(R)\in\Re^3$, whose $i$-th element is given by
\begin{equation}
    [\nabla \mathcal{C}(R)]_i = \frac{1}{4\pi}\sum_{l=0}^{B-1} (G^l)^T U^l(R) u^l(e_i) F^l,\label{eqn:cor_grad}
\end{equation}
for $i\in\{1,2,3\}$.
In the above expression, we have used the homomorphism property $U^l(R\exp(\epsilon\hat\eta)) = U^l(R) U^l(\exp(\epsilon\hat\eta))$. 
Using these, one can apply any gradient-based numerical optimization algorithm, such as one summarized in Table~\ref{tab:SSM}.

\subsection{Earth Elevation Map}

Here we consider a particular example with a worldwide elevation model, namely ETOPO5~\cite{ETOPO5}. 
The topological elevation model is interpolated such that for a given set of latitude and longitude the corresponding elevation is calculated. 
This function yields $f(x)$, which is transformed by spherical harmonics with the bandwidth of $B=129$ to obtain $F^l$. 
We choose a particular rotation matrix with 3-2-3 Euler angle $(\alpha,\beta,\gamma) = (\frac{\pi}{6},\frac{\pi}{3},\frac{\pi}{4})$, and perform the Fourier transform of $f(R^Tx)$ to obtain $G^l$. 
See Figure~\ref{fig:earth} for the original elevation map, and the rotated one. 

The optimization algorithm summarized in Table~\ref{tab:SSM} is implemented with the initial guess $(\alpha,\beta,\gamma)=(0.3,0.3,0.3)$, a tolerance $\epsilon=10^{-6}$ and step-size $\delta=5\times 10^{-3}$. 
The numerical optimization is terminated after 223 iterations, with the maximum absolute error of $7.98\times 10^{-5}$ in terms of Euler angles.
The evolution of the correlation function and its gradient magnitude over iterations, and the change of the rotation matrix in terms of Euler angles are illustrated in Figure~\ref{fig:iter}.

\newcommand{\algrule}[1][.2pt]{\par\vskip.2\baselineskip\hrule height #1\par\vskip.2\baselineskip}

\begin{table}
    \caption{Spherical shape matching algorithm}\label{tab:SSM}
    \begin{algorithmic}[1]
        \algrule[0.8pt]
        \Procedure{Spherical shape matching}{}
        \algrule
        \State Perform Fourier transform of $f(R),g(R)$ to obtain $F^l,G^l$
        \State Make an initial guess of $R$
        \State Set a tolerance $\epsilon>0$ and a step size $\delta >0$
        \Repeat 
        \State $(\mathcal{C},\nabla\mathcal{C})$={\fontshape{sc}\selectfont Correlation}{$(R)$}
        \State Update $R = R \exp(\delta \widehat{\nabla\mathcal{C}})$
        \Until{$\|\nabla\mathcal{C}(R)\|< \epsilon $}
        \State \Return $R$
        \EndProcedure
        \algrule
        \Procedure{$(\mathcal{C},\nabla\mathcal{C})$=Correlation}{$R$}
        \State Compute $\mathcal{C}$ with \eqref{eqn:cor}
        \State Compute $\nabla\mathcal{C}$ with \eqref{eqn:cor_grad}
        \EndProcedure
       \algrule[0.8pt]
    \end{algorithmic}
\end{table}

\begin{figure}
    \centerline{
        \subfigure[Original Earth elevation map]{
            \includegraphics[trim={6cm 3cm 5.5cm 3cm},clip,height=0.38\textwidth]{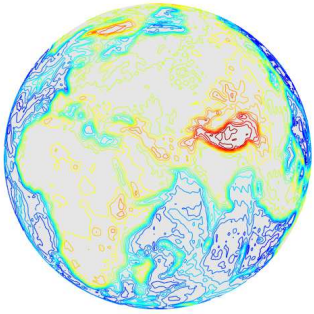}
        }
        \hspace*{0.1\textwidth}
        \subfigure[Rotated Earth elevation map]{
            \includegraphics[trim={6cm 3cm 5.5cm 3cm},clip,height=0.38\textwidth]{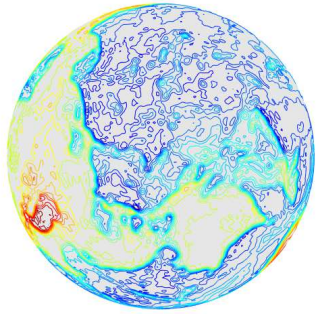}
        }
    }
    \caption{Application to spherical image matching for Earth elevation map}\label{fig:earth}
\end{figure}

\begin{figure}
    \centerline{
        \subfigure[Evolution of correlation function and its gradient over iteration]{
            \includegraphics[height=0.34\textwidth]{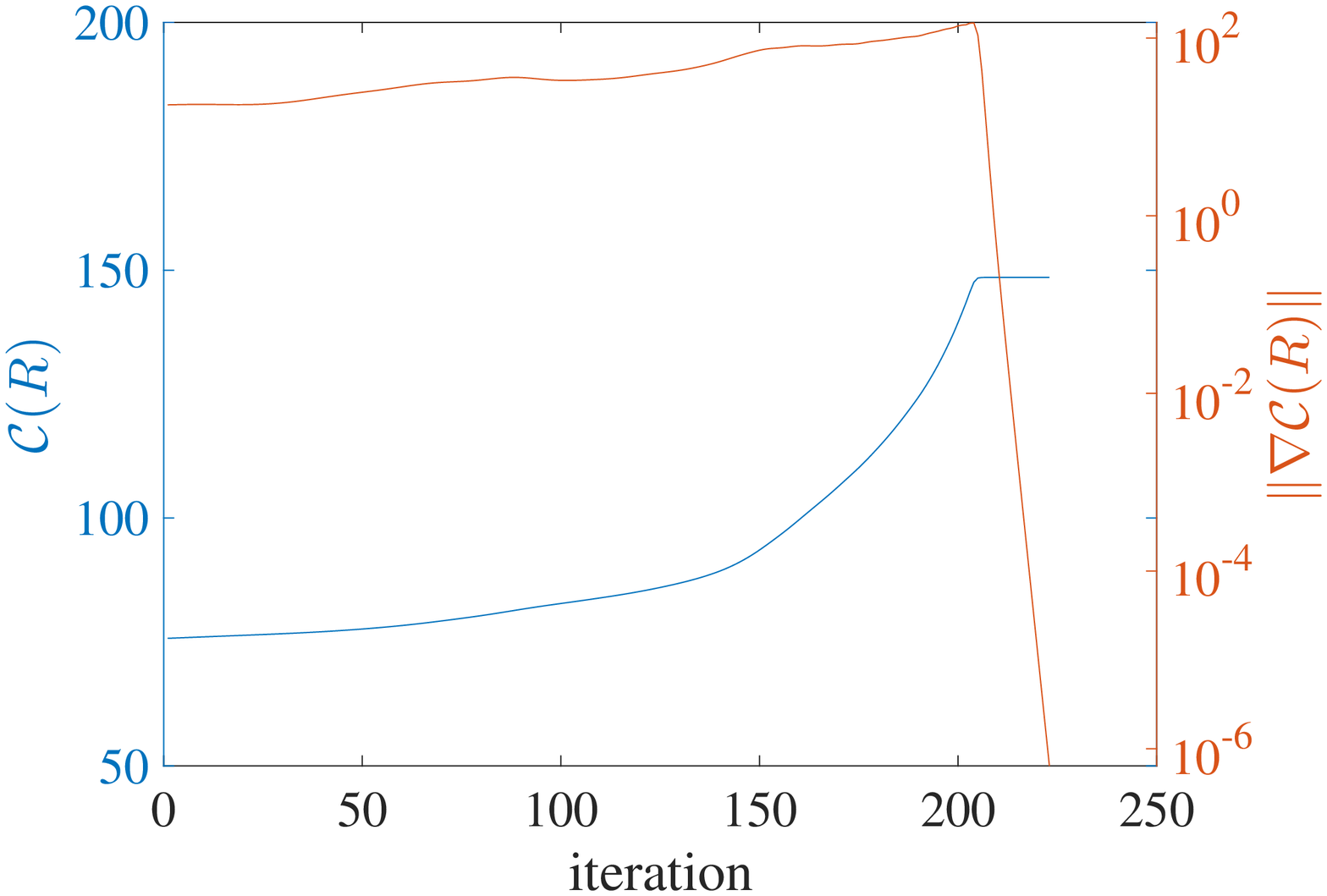}
        }
        \hspace*{0.0\textwidth}
        \subfigure[Evolution of Euler angles over iteration]{
            \includegraphics[height=0.34\textwidth]{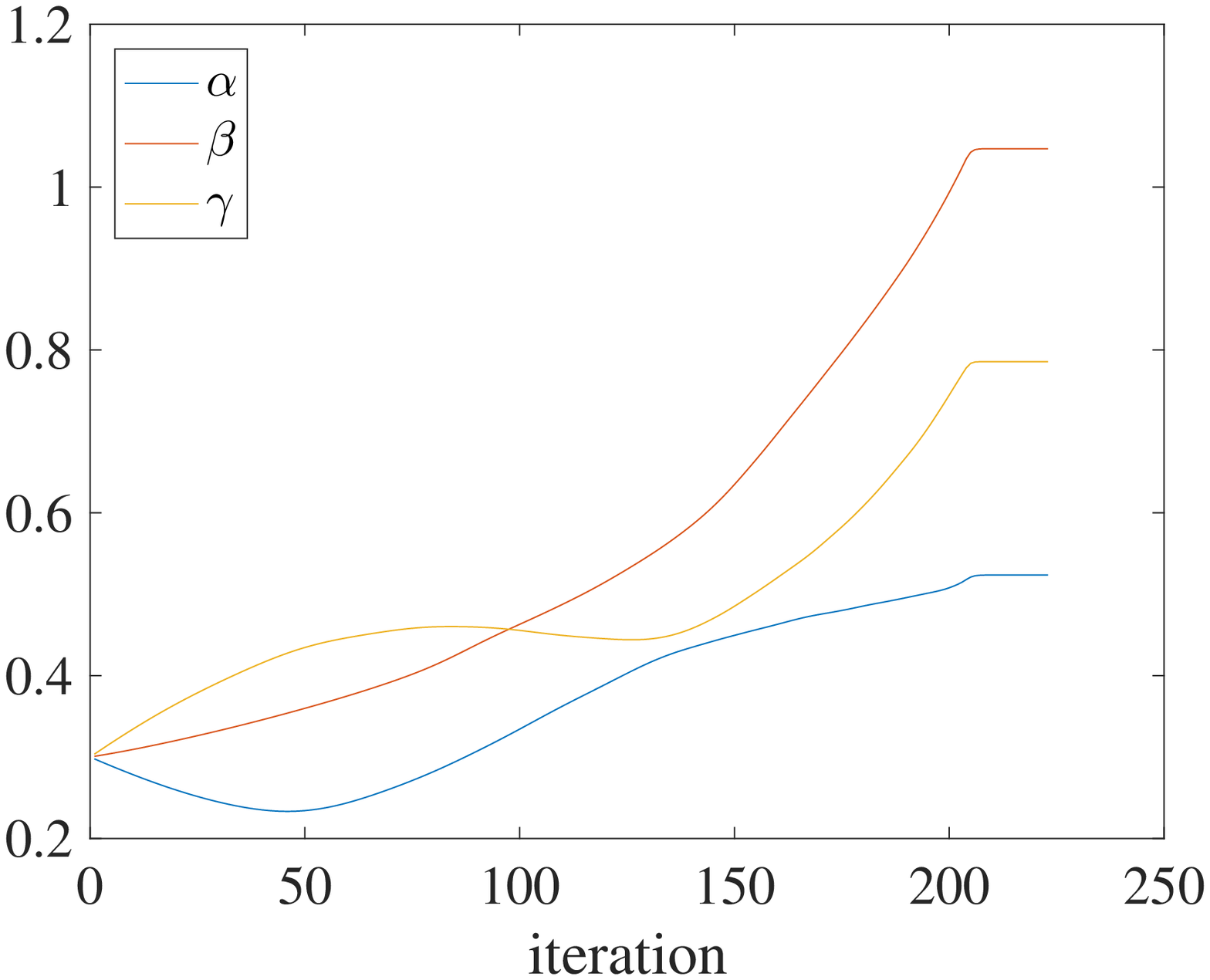}
        }
    }
    \caption{Iteration procedure for spherical shape matching of Earth elevation map}\label{fig:iter}
\end{figure}

\section{Conclusions}

We have presented real harmonic analysis on the special orthogonal group, including various operational properties such as fast Fourier transform, Clebsch-Gordon coefficients, and derivatives. 
There are further implemented into an open source software package supporting parallel processing for accelerated computing. 
In particular, the presented form of real irreducible, unitary representations can be evaluated without constructing their counter parts in complex harmonic analysis, namely wigner D-matrices, and the given formulation in terms of Euler angles are suitable for fast Fourier transform. 

Future works include extension beyond the special orthogonal group, such as the special Euclidean group for various engineering applications regarding the coupled translational and rotational dynamics of a rigid body. 
Also, the intermediate term $\Psi^l_{l,m}(\beta)$ in~\eqref{eqn:Psi} may be directly evaluated without formulating real-valued wigner d-functions.

\begin{acknowledgements}
    This research has been supported in parts by NSF under the grant CMMI-1335008, and by AFOSR under the grant FA9550-18-1-0288.
\end{acknowledgements}


\end{document}